\documentclass[11pt,reqno]{amsart}
\usepackage{color}
\usepackage[english]{babel}
\usepackage{amsfonts}
\usepackage{amsmath}
\usepackage{amsthm}
\usepackage{amssymb}
\usepackage[misc]{ifsym}
\usepackage{bbm}
\usepackage{mathrsfs}
\usepackage{esint}
\usepackage{faktor}
\usepackage[utf8]{inputenc}
\usepackage{afterpage}
\usepackage[margin=0.90in]{geometry}

\usepackage{graphicx}
\usepackage{enumitem}
\usepackage[dvipsnames]{xcolor}
\usepackage[colorlinks=true,urlcolor=blue,citecolor=blue,linkcolor=blue,linktocpage,pdfpagelabels,bookmarksnumbered,bookmarksopen]{hyperref}

\newcommand{\R}{\mathbb{R}}

\newcommand{\Div}{\mathrm{div} \, }
\newcommand{\dx}{\, {\rm d} x}

\newcommand{\dz}{\, {\rm d} z}
\newcommand{\dt}{\, {\rm d} t}
\newcommand{\ds}{\, {\rm d} s}

\newcommand{\dtau}{\, {\rm d} \tau}
\newcommand{\eps}{\varepsilon}
\newcommand{\loc}{{\rm loc}}

\newcommand{\Jac}{{\rm Jac}}

\newtheorem{lemma}{Lemma}
\newtheorem{thm}[lemma]{Theorem}

\theoremstyle{definition}

\DeclareMathOperator*{\diam}{diam}

\begin{document}
\title[On the linear independence condition]{On the linear independence condition \\ for the Bobkov-Tanaka first eigenvalue \\ of the double-phase operator}
\author[N. Biswas]{Nirjan Biswas}
\address[N. Biswas]{Department of Mathematics, Indian Institute of Science Education and Research Pune, Dr. Homi Bhabha Road, Pune 411008, India}
\email{nirjan.biswas@acads.iiserpune.ac.in}
\author[L. Gambera]{Laura Gambera}
\address[L. Gambera]{Dipartimento di Matematica e Informatica, Universit\`a degli Studi di Catania, Viale A. Doria 6, 95125 Catania, Italy}
\email{laura.gambera@unipa.it}
\author[U. Guarnotta]{Umberto Guarnotta}
\address[U. Guarnotta]{Dipartimento di Ingegneria Industriale e Scienze Matematiche, Università Politecnica delle Marche, Via Brecce Bianche 12, 60131 Ancona, Italy}
\email{u.guarnotta@univpm.it}

\begin{abstract}
The paper investigates a pivotal condition for the Bobkov-Tanaka type spectrum for double-phase operators. This condition is satisfied if either the weight $w$ driving the double-phase operator is strictly positive in the whole domain or the domain is convex and fulfils a suitable symmetry condition.
\end{abstract}

\maketitle

\let\thefootnote\relax
\footnote{{\bf{MSC 2020}}: 35J60, 35P30, 35B44, 35B40.}
\footnote{{\bf{Keywords}}: double-phase operator, non-homogeneous spectrum, blow-up arguments.}
\footnote{\Letter \quad Corresponding author: Umberto Guarnotta (u.guarnotta@univpm.it).}

\vspace{-1cm}

\section{Introduction}

In \cite{GG}, the authors investigated the existence of positive solutions to
\begin{equation}
\label{prob}
\left\{
\begin{alignedat}{2}
-\Delta_{p}^w u -\Delta_{q}u&= \alpha  w(x)|u|^{p-2}u+ \beta |u|^{q-2}u &&\quad \mbox{in}\;\; \Omega, \\
u &=0 &&\quad \mbox{on}\;\; \partial\Omega, 
\end{alignedat}
\right.
\end{equation}
where $\Omega\subseteq \R^N$, $N\geq 2$, is a bounded domain of class $C^2$, $1<q<p<N$ satisfies $p<q^*$ ($q^*:=\frac{Nq}{N-q}$ is the critical exponent), $\alpha,\beta\in\R$, and $\Delta_p^w+\Delta_q$ is the double-phase operator driven by the positive weight $w\in C^{0,1}(\overline{\Omega})$ belonging to the $p$-Muckenhoupt class $A_p$ (for details, see the end of this section).

Different notions of spectrum for non-homogeneous operators have been introduced: for a brief account, we address the reader to \cite{GG}. The solutions to \eqref{prob} are the eigenfunctions of the Bobkov-Tanaka type spectrum for the double phase operator, where the couple $(\alpha,\beta)$ represents the eigenvalue corresponding to $u$. This spectrum, introduced ten years ago for the $(p,q)$-Laplacian \cite{BT,BT2,BT3}, has been recently extended in the non-local case \cite{BS} and for the double-phase operator \cite{GG}.

Let $w\in C^{0,1}(\overline{\Omega})$ be a non-negative function. We consider the weighted $p$-Laplacian eigenvalue problem
\begin{equation*}
\left\{
\begin{alignedat}{2}
-\Delta_{p}^w v &=  \lambda w(x)|v|^{p-2}v &&\quad \mbox{in}\;\; \Omega, \\
v &=0 &&\quad \mbox{on}\;\; \partial\Omega. 
\end{alignedat}
\right.
\end{equation*}
Under reasonable conditions on $w$ (for instance $w\in A_p$; see \cite{PPR}), there exists a smallest eigenvalue $\lambda_p^w$, which is positive, simple, and isolated, with corresponding ($L^\infty$-normalized) positive eigenfunction $\phi_p^w\in L^\infty(\Omega)$. Moreover, we denote with $(\lambda_q,\phi_q)$ the first ($L^\infty$-normalized) eigenpair (see \cite{Le}) of the $q$-Laplacian eigenvalue problem
\begin{equation*}
\left\{
\begin{alignedat}{2}
-\Delta_{q} z &=  \mu |z|^{q-2}z &&\quad \mbox{in}\;\; \Omega, \\
z &=0 &&\quad \mbox{on}\;\; \partial\Omega.
\end{alignedat}
\right.
\end{equation*}
The existence of solutions to \eqref{prob} is strongly influenced by the so-called `linear independence condition', namely,
\begin{equation}
\label{LI}
\tag{LI}
\phi_p^{w}\neq \phi_q.
\end{equation}
In this paper we prove that \eqref{LI} holds for any domain when the weight $w\in A_p$ is strictly positive, and the strict-positivity hypothesis can be removed when working in convex domains having a particular symmetry. Our results pertain to regular weights (i.e., of class $C^{0,1}$ or $C^1$).

 The double-phase operator is widely applied in various real-life contexts, including fluid dynamics, biology, materials science, and game theory. For example, the double-phase operator models the flow behaviour of non-Newtonian fluids, capturing phenomena such as viscosity changes due to shear thinning or thickening (see \cite{SS, CM, BCM}). Moreover, it can also be used to model materials with heterogeneous properties: indeed, the coefficient $w(x) \ge 0$ dictates the geometry of the composite made of two different materials, one with a hardening exponent $p$ and the other with a hardening exponent $q$ (vide \cite{RZ} and the references therein). Finally, double-phase operators can be employed to enhance image decomposition and denoising \cite{BCZ}.

We recall that a function $w$ belongs to the $p$-Muckenhoupt class $A_p$ if $w,w^{-1}\in L^1_\loc(\Omega)$ and
\begin{equation}
\label{Muck}
\frac{1}{|B|} \int_B w \dx \leq C \left(\frac{1}{|B|} \int_B w^{\frac{1}{1-p}} \dx \right)^{1-p}
\end{equation}
for some $C>0$ and any ball $B\subseteq \Omega$; see \cite[p.297]{HKM}.
The double-phase operator is defined as
$$ \Delta_p^w u + \Delta_q u := \Div(w|\nabla u|^{p-2}\nabla u + |\nabla u|^{q-2}\nabla u) \quad \mbox{for all} \;\; u\in W^{1,\Theta}(\Omega), $$
being $W^{1,\Theta}(\Omega)$ the Musielak-Orlicz space (see \cite{HH}) associated with the $N$-function
$$ \Theta(x,t) := w(x)|t|^p+|t|^q \quad \mbox{for all} \;\; (x,t)\in\Omega\times\R.$$

Given $f,g:\Omega\to\R$ and $x_0\in\Omega$, we write $f\sim g$ as $x\to x_0$ if there exist $c_1,c_2>0$ such that
$$ c_1 \leq \liminf_{x\to x_0} \frac{f(x)}{g(x)} \leq \limsup_{x\to x_0} \frac{f(x)}{g(x)} \leq c_2. $$

Moreover, for every $\lambda>0$, $A\subseteq\R^N$, and $u:A\to\R$ we define
$$ A_\lambda := \lambda^{-1}A = \{\lambda^{-1}x: \, x\in A\}, \; u_\lambda(y):= u(\lambda y) \quad \mbox{for all} \;\; y\in A_\lambda. $$

Finally, following \cite{GM}, we say that $\Omega$ is symmetric if there exist $N-1$ two-dimensional planes $T_k$ and $N-1$ rotations of $\theta_k\neq 0,\pi$ radians, with $k=1,\ldots,N-1$, such that $\Omega$ is invariant under the rotation $\theta_k$ in the plane $T_k$ for all $k=1,\ldots,N-1$.

\section{Main results}

Hereafter, we fix any $p,q>0$ such that $1<q<p<N$ with $p<q^*$.

\begin{thm}
\label{mainthm1}
Let $w\in C^{0,1}(\overline{\Omega})\cap A_p$ be strictly positive in $\Omega$. Then \eqref{LI} holds true.
\end{thm}
\begin{proof}
For the sake of brevity, set $\phi:=\phi_p^w$ and $\lambda_1:=\lambda_p^w$. Arguing by contradiction, we suppose $\phi=\phi_q$. Since $\phi_q\in C^{1,\tau}(\overline{\Omega})$, then also $\phi\in C^{1,\tau}(\overline{\Omega})$, and in particular $\phi$ admits a global maximizer in $\Omega$ by Weierstrass' theorem, together with $\phi>0$ in $\Omega$ and $\phi=0$ on $\partial\Omega$. Without loss of generality, assume that $0\in\Omega$ and $\phi(0)=\max_\Omega \phi$. Given any $\lambda>0$, we put
$$ \hat{\phi}_\lambda(y):=\frac{\phi(0)-\phi_\lambda(y)}{\lambda^{p'}} \quad \mbox{for all} \;\; y\in\Omega_\lambda. $$
It is readily seen that
\begin{equation}
\label{blowup}
\Delta_p^{w_\lambda} \hat{\phi}_\lambda = \lambda_1 w_\lambda \phi_\lambda^{p-1} \quad \mbox{in} \;\; \Omega_\lambda.
\end{equation}
We claim that $\{\hat{\phi}_\lambda: \, \lambda>0\}$ is bounded in $C^{1,\tau}_\loc(\R^N)$, that is, for any compact set $K\subseteq \R^N$ there exist $C_K,\Lambda_K>0$ such that $\|\hat{\phi}_\lambda\|_{C^{1,\tau}(K)} \leq C_K$ for all $\lambda\in(0,\Lambda_K)$. We divide the proof in the following steps.

\begin{enumerate}[label={\arabic*.}]
\item \underline{Localization}. Fix any compact $K\subseteq \R^N$ and $H\Subset \Omega$ open neighborhood of $x=0$. Choose $R_K,\Lambda_K>0$ such that $K\Subset B_{R_K} \Subset H_\lambda$ for all $\lambda\in(0,\Lambda_K)$ (notice that $H_\lambda \nearrow \R^N$ as $\lambda \to 0$). Further, since $H_{\lambda} \Subset \Omega_{\lambda}$ for each $\lambda>0$, we indeed have $B_{R_K} \Subset \Omega_{\lambda}$ for all $\lambda\in(0,\Lambda_K)$. Owing to $w\in C^0(\overline{\Omega})$ and $w>0$ in $\Omega$, there exist $\gamma_H,\Gamma_H>0$ such that $\gamma_H \leq w\leq \Gamma_H$ in $H$. We will apply the local arguments in steps 2-3 to the equation \eqref{blowup} restricted to $B_{R_K}$ so that our estimates will be independent of $\lambda$ (they will depend only on $R_K$ and $\Lambda_K$, that is, on $K$). 
\item \underline{Local $L^\infty$ estimate}. Preliminarily, notice that $\hat{\phi}_\lambda(0)=0$ and $\hat{\phi}_\lambda$ is non-negative, since $\phi(0)$ is the global maximum for $\phi$. Moreover, $\gamma_H \leq w_\lambda \leq \Gamma_H$ in $H_\lambda$, so the operator in \eqref{blowup} satisfies
\begin{equation*}
|w_\lambda(y) |\xi|^{p-2}\xi| \leq \Gamma_H |\xi|^{p-1} \quad \mbox{and} \quad \langle w_\lambda(y)|\xi|^{p-2}\xi,\xi\rangle \geq \gamma_H |\xi|^p \quad \mbox{for all} \;\; (y,\xi)\in H_\lambda \times \R^N.
\end{equation*}
Concerning the right-hand side of \eqref{blowup}, we observe that
\begin{equation*}
|\lambda_1 w_\lambda(y) \phi_\lambda(y)^{p-1}| \leq \lambda_1 \|w\|_{C^0(\overline{\Omega})} \|\phi\|^{p-1}_{C^0(\overline{\Omega})} \quad \mbox{for all} \;\; y\in H_\lambda.
\end{equation*}
Accordingly, the Harnack inequality \cite[Theorem 7.2.2]{PS} and $\inf_{B_R} \hat{\phi}_\lambda = 0$ yield
\begin{equation*}
\sup_K \hat{\phi}_\lambda \leq \sup_{B_{R_K}} \hat{\phi}_\lambda \leq C\left(\inf_{B_{R_K}} \hat{\phi}_\lambda + 1\right) = C \quad \mbox{for all} \;\; \lambda\in(0,\Lambda_K),
\end{equation*}
for a suitable $C>0$ depending on $p,N,\Omega,\|w\|_{C^0(\overline{\Omega})},\gamma_H,\Gamma_H,K$.
\item \underline{$C^{1,\alpha}$ local regularity}. Notice that
$$ \Jac(w_\lambda(y)|\zeta|^{p-2}\zeta) = w_\lambda(y)\left[(p-2)|\zeta|^{p-4}\zeta\otimes\zeta + |\zeta|^{p-2}I\right] \quad \mbox{for all} \;\; (y,\zeta)\in\Omega_\lambda\times\R^N, $$
where $\Jac$ is the Jacobian with respect to $\zeta$, $I$ is the identity matrix, while $\otimes$ denotes the tensor product. As a consequence, denoting with $\|\cdot\|$ any matrix norm in $\R^N$,
\begin{equation*}
\begin{aligned}
\langle \Jac(w_\lambda(y)|\zeta|^{p-2}\zeta)\xi,\xi\rangle &= w_\lambda(y)\left[(p-2)|\zeta|^{p-4}|\zeta \cdot \xi|^2 + |\zeta|^{p-2}|\xi|^2\right] \geq \gamma_H \min\{p-1,1\} |\zeta|^{p-2}|\xi|^2, \\
\|\Jac(w_\lambda(y)|\zeta|^{p-2}\zeta)\| &\leq w_\lambda(y)(p-1)|\zeta|^{p-2} \leq \Gamma_H (p-1)|\zeta|^{p-2},
\end{aligned}
\end{equation*}
for all $(y,\zeta,\xi)\in H_\lambda\times\R^N\times\R^N$. In addition, due to $w\in C^{0,1}(\overline{\Omega})$,
\begin{equation*}
|w(y_1)|\zeta|^{p-2}\zeta-w(y_2)|\zeta|^{p-2}\zeta| \leq [w]_{C^{0,1}(\overline{\Omega})}|y_1-y_2||\zeta|^{p-1} \quad \mbox{for all} \;\; y_1,y_2\in H_\lambda, \;\; \zeta\in\R^N.
\end{equation*}
Since $\|\hat{\phi}_\lambda\|_{L^\infty(K)} \leq C$ by step 2, \cite[Theorem 1.7]{L} entails
\begin{equation*}
\|\hat{\phi}_\lambda\|_{C^{1,\tau}(K)} \leq \hat{C} \quad \mbox{for all} \;\; \lambda \in (0,\Lambda_K),
\end{equation*}
for an opportune $\hat{C}>0$ depending on $p,N,\Omega,\|w\|_{C^{0,1}(\overline{\Omega})},\gamma_H,\Gamma_H,K$. Arbitrariness of $K\Subset\R^N$ implies that $\{\hat{\phi}_\lambda\}$ is bounded in $C^{1,\tau}_\loc(\R^N)$; thus, Ascoli-Arzelà's theorem and a diagonal argument ensure that $\hat{\phi}_\lambda \to \overline{\phi}$ in $C^1_\loc(\R^N)$ as $\lambda\to 0$, for a suitable $\overline{\phi}\in C^1_\loc(\R^N)$. In particular, $\overline{\phi}(0)=0$.
\end{enumerate}

Next, we claim that, if $f\in C^{0,\eta}(\overline{\Omega})$ for some $\eta\in(0,1]$, then $f_\lambda \to f(0)$ in $C^0_\loc(\R^N)$. Indeed, given $K\Subset \R^N$ and selecting any $R_K,\Lambda_K>0$ such that $K\Subset B_{R_K} \Subset \Omega_\lambda$ for all $\lambda \in (0,\Lambda_K)$, one has
\begin{equation*}
\sup_{y\in K} |f(\lambda y)-f(0)| \leq [f]_{C^{0,\eta}(\overline{\Omega})}|\lambda y|^{\eta} \leq [f]_{C^{0,\eta}(\overline{\Omega})}(\diam K)^{\eta} \lambda^\eta \to 0 \quad \mbox{as} \;\; \lambda\to 0,
\end{equation*}
proving the claim.

Since $w\in C^{0,1}(\overline{\Omega})$ and $\phi\in C^{1,\tau}(\overline{\Omega})$, the above claims ensure both $w_\lambda \to w(0)$ and $\phi_\lambda \to \phi(0)$ in $C^0_\loc(\R^N)$. Hence, passing to the limit in \eqref{blowup} via local uniform convergence, we get
\begin{equation*}
\Div(w(0)|\nabla\overline{\phi}|^{p-2}\nabla\overline{\phi}) = \lambda_1 w(0)\phi(0)^{p-1} \quad \mbox{in} \;\; \R^N,
\end{equation*}
namely,
\begin{equation}
\label{limiteq}
\Delta_p \overline{\phi} = \lambda_1 \phi(0)^{p-1} \quad \mbox{in} \;\; \R^N,
\end{equation}
which is (A.2) of \cite{BT2}. Then the proof follows verbatim \cite[Proposition A.1]{BT2}, leading to $\overline{\phi}=0$ in $\R^N$, in sharp contrast with \eqref{limiteq}.
\end{proof}

We observe that the proof of Theorem \ref{mainthm1} strongly relies on the positivity of $w$ on a local maximizer of $\phi_q$. This condition is unnecessary for special domains and sufficiently regular weights, as the following theorem shows.

\begin{thm}
\label{mainlemma}
Let $\Omega$ be convex and symmetric. Then \eqref{LI} is satisfied for any $w\in C^1(\overline{\Omega})\cap A_p$.
\end{thm}
\begin{proof}
Since $\Omega$ is convex and symmetric, \cite[Theorem 3]{GM} ensures that $\phi_q$ has a unique critical point (hence, a global and strict maximizer) $x_0\in\Omega$ and $\phi_q\in C^{1,\tau}(\overline{\Omega})\cap C^2(\overline{\Omega}\setminus\{x_0\})$. Without loss of generality, we assume that $x_0=0$. By contradiction, suppose that \eqref{LI} does not hold true. Hence $\phi:=\phi_p^w=\phi_q\in C^{1,\tau}(\overline{\Omega})\cap C^2(\overline{\Omega}\setminus\{0\})$ satisfies
\begin{equation*}
-\Delta_p^w \phi = \lambda_p^w w \phi^{p-1} \quad \mbox{in} \;\; \Omega, \quad -\Delta_q \phi = \lambda_q \phi^{q-1} \quad \mbox{in} \;\; \Omega, \quad \phi=0 \quad \mbox{on} \;\; \partial\Omega.
\end{equation*}
Set $\sigma:=w |\nabla \phi|^{p-q}\in C^0(\overline{\Omega})\cap C^1(\overline{\Omega}\setminus \{0\})$. Then
\begin{equation*}
\begin{aligned}
\lambda_p^w w \phi^{p-1} &= -\Delta_p^w \phi = -\Div(\sigma |\nabla \phi|^{q-2}\nabla \phi) = -|\nabla \phi|^{q-2}\nabla \phi \cdot \nabla \sigma -\sigma \Delta_q \phi \\
&= -|\nabla \phi|^{q-2}\nabla \phi \cdot \nabla \sigma + \lambda_q \phi^{q-1}\sigma
\end{aligned}
\end{equation*}
everywhere in $\Omega\setminus \{0\}$. Dividing by $|\nabla \phi|^q$, we obtain the first-order PDE
\begin{equation}
\label{normalized}
\frac{\nabla \phi}{|\nabla \phi|^2} \cdot \nabla \sigma + \frac{1}{|\nabla \phi|}\left[\lambda_p^w \left(\frac{\phi}{|\nabla \phi|}\right)^{p-1} - \lambda_q \left(\frac{\phi}{|\nabla \phi|}\right)^{q-1}\right] \sigma = 0 \quad \mbox{everywhere in}\;\; \Omega\setminus \{0\}.
\end{equation}
We set $E:=\frac{\nabla \phi}{|\nabla \phi|^2}$ and $g:=\frac{1}{|\nabla \phi|}\left[\lambda_p^w \left(\frac{\phi}{|\nabla \phi|}\right)^{p-1} - \lambda_q \left(\frac{\phi}{|\nabla \phi|}\right)^{q-1}\right]$, so that \eqref{normalized} becomes
\begin{equation}
\label{pde}
E(x)\cdot \nabla \sigma + g(x)\sigma = 0 \quad \mbox{for all}\;\; x\in\Omega\setminus \{0\}.
\end{equation}
According to \cite[Theorem 1]{GM}, the following asymptotic estimates hold true:
\begin{equation}
\label{asympest}
\phi(0)-\phi(x) \sim |x|^{q'} \quad \mbox{and} \quad |\nabla \phi(x)| \sim |x|^{\frac{1}{q-1}}, \quad \mbox{when} \;\; |x|\to 0.
\end{equation}
As a consequence, recalling that $\phi(0)>0$ and $p>q$,
\begin{equation*}
g(x) \sim |x|^{-\frac{p}{q-1}}-|x|^{-q'} \sim |x|^{-\frac{p}{q-1}} \quad \mbox{as} \;\; |x|\to 0.
\end{equation*}
Thus, there exist $k_1,k_2,\eps>0$ such that
\begin{equation}
\label{asympests}
\begin{alignedat}{2}
&k_1|x|^{q'} \leq \phi(0)-\phi(x) \leq k_2|x|^{q'} \quad &&\mbox{for all} \;\; x\in B_\eps, \\
&k_1|x|^{\frac{1}{q-1}} \leq |\nabla \phi(x)| \leq k_2|x|^{\frac{1}{q-1}} \quad &&\mbox{for all} \;\; x\in B_\eps, \\
&k_1|x|^{-\frac{p}{q-1}}\leq g(x)\leq k_2|x|^{-\frac{p}{q-1}} \quad &&\mbox{for all} \;\; x\in B_\eps.
\end{alignedat}
\end{equation}
Choose any $l\in(0,\phi(0))$ such that $\phi^{-1}([l,\phi(0)])\subseteq B_\eps$ (which is possible since $\phi\in C^0(\overline{\Omega})$ and $0$ is the strict global maximizer of $\phi$). From \eqref{asympests} we infer $B_\rho\subseteq \phi^{-1}([l,\phi(0)])$, where $\rho:=\left(\frac{\phi(0)-l}{k_2}\right)^{\frac{1}{q'}}$. Now pick any $\xi\in\phi^{-1}(l)$ and consider the Cauchy problem
\begin{equation}
\label{characteristics}
\left\{
\begin{alignedat}{2}
&\dot x = E(x), \\
&x(0) = \xi.
\end{alignedat}
\right.
\end{equation}
Since $\xi\neq 0$ and $E\in C^{0,1}_\loc(\Omega\setminus\{0\})$, \eqref{characteristics} possesses a unique global solution $x=x(t)$, defined in its maximal interval $(0,T)$, being $T\in(0,+\infty]$. Observe that $V:\Omega\to\R$ defined as $V(x):=\phi(0)-\phi(x)$ is a Lyapunov function for \eqref{characteristics}: indeed, it is readily seen that $V\geq 0$ in $\Omega$, $V(x)=0$ if and only if $x=0$, and $V$ is decreasing along the flow generated by \eqref{characteristics}, since
\begin{equation}
\label{monotoneflow}
\dot V = \nabla V \cdot \dot x = -\nabla \phi \cdot \frac{\nabla \phi}{|\nabla \phi|^2} = -1 < 0 \quad \mbox{in} \;\; (0,T).
\end{equation}
Thus, $x=0$ is an asymptotically stable equilibrium, that is,
\begin{equation}
\label{asympstable}
\lim_{t\to T^-} x(t) = 0.
\end{equation}
Due to \eqref{monotoneflow} and the choice of $l$, $x(t)\in B_\eps$ for all $t\in(0,T)$. Moreover, $T=\phi(0)-l$: indeed, let $t\to T^-$ and use \eqref{asympstable} in
$$ \phi(x(t))-l = \phi(x(t))-\phi(x(0)) = \int_0^t \nabla \phi(x(\tau))\cdot \dot{x}(\tau) \dtau = \int_0^t \dtau = t \quad \mbox{for all}\;\; t\in(0,T). $$
Set $\theta:=\sigma \circ x$. Then $\dot \theta = \nabla \sigma \cdot \dot x = E(x)\cdot \nabla \sigma$ in $(0,T)$, so \eqref{pde} can be written, along the characteristic curve $x=x(t)$,
\begin{equation}
\label{ode}
\dot \theta + g(x)\theta = 0 \quad \mbox{in} \;\; (0,T).
\end{equation}
Therefore, the non-trivial solution to \eqref{ode} can be explicitly written as
\begin{equation}
\label{odesol}
\theta(t) = \theta(0)\exp\left(-\int_0^t g(x(\tau)) \dtau\right) = \sigma(\xi)\exp\left(-\int_0^t g(x(\tau)) \dtau\right) \quad \mbox{for all} \;\; t\in [0,T).
\end{equation}
We perform the change of variable $z=|x(\tau)|$ in \eqref{odesol}, whence $\dz = \frac{x(\tau)}{|x(\tau)|}\cdot \dot x(\tau) \dtau = \frac{x(\tau)}{|x(\tau)|}\cdot \frac{\nabla \phi(x(\tau))}{|\nabla \phi(x(\tau))|^2} \dtau$. Thus, for any $t\in(0,T)$, the continuity of $g$ jointly with \eqref{asympests} yields
\begin{equation*}
\begin{aligned}
\int_0^t g(x(\tau)) \dtau &\geq k_1\int_0^t |x(\tau)|^{-\frac{p}{q-1}} \dtau \geq k_1^2\int_0^t \frac{|x(\tau)|^{-\frac{p-1}{q-1}}}{|\nabla\phi(x(\tau))|} \dtau \\
&\geq k_1^2\int_0^t |x(\tau)|^{-\frac{p-1}{q-1}} \left[-\frac{x(\tau)}{|x(\tau)|}\cdot \frac{\nabla \phi(x(\tau))}{|\nabla \phi(x(\tau))|^2}\right] \dtau \\
&= -k_1^2\int_{|\xi|}^{|x(t)|} z^{-\frac{p-1}{q-1}} \dz = k_1^2\,\frac{q-1}{p-q}\left[|x(t)|^{-\frac{p-q}{q-1}}-|\xi|^{-\frac{p-q}{q-1}}\right].
\end{aligned}
\end{equation*}
Since, varying with $\xi$, the characteristics \eqref{characteristics} neither intersect nor auto-intersect, one can consider $t=t(x):\phi^{-1}([l,\phi(0)))\to[0,T)$ defined as the inverse function of the characteristic \eqref{characteristics}. Then, setting $M:=k_1^2\,\frac{q-1}{p-q}$ and recalling \eqref{asympests}, we get $|\xi|\geq\rho$ and
\begin{equation}
\label{sigmaasymp}
\begin{aligned}
\sigma(x) &= \theta(t(x)) = \sigma(\xi)\exp\left(-\int_0^{t(x)} g(x(\tau)) \dtau\right) \leq \sigma(\xi)\exp\left(M|\xi|^{-\frac{p-q}{q-1}}t(x)\right)\exp\left(-M|x|^{-\frac{p-q}{q-1}}\right) \\
&\leq \left(\max_{\overline{\Omega}} \sigma\right) \exp\left(MT\rho^{-\frac{p-q}{q-1}}\right)\exp\left(-M|x|^{-\frac{p-q}{q-1}}\right) \quad \mbox{for all}\;\; x\in B_\rho\setminus\{0\}.
\end{aligned}
\end{equation}
Now we show that $w\notin A_p$. Exploiting \eqref{asympests} and \eqref{sigmaasymp} we have
\begin{equation*}
w(x) = \sigma(x) |\nabla u(x)|^{q-p} \leq c |x|^{-\frac{p-q}{q-1}}\exp\left(-M|x|^{-\frac{p-q}{q-1}}\right) \quad \mbox{for all} \;\; x\in B_\rho\setminus\{0\},
\end{equation*}
being $c>0$ a suitable constant. Hence, integrating in polar coordinates and using the change of variable $t=\frac{M}{p-1}\,s^{-\frac{p-q}{q-1}}$, for every $r\in(0,\rho)$ we get
\begin{equation*}
\begin{aligned}
\fint_{B_r} w^{-\frac{1}{p-1}} \dx &\geq cr^{-N} \int_{B_r} |x|^{\frac{p-q}{(p-1)(q-1)}}\exp\left(\frac{M}{p-1}\,|x|^{-\frac{p-q}{q-1}}\right) \dx \\
&= cr^{-N} \int_0^r s^{N-1+\frac{p-q}{(p-1)(q-1)}}\exp\left(\frac{M}{p-1}\,s^{-\frac{p-q}{q-1}}\right) \ds \\
&= cr^{-N} \int_{\frac{M}{p-1}\,r^{-\frac{p-q}{q-1}}}^{+\infty} t^{-\left(N\frac{q-1}{p-q}+p'\right)}e^t \dt = +\infty,
\end{aligned}
\end{equation*}
being $c>0$ a small constant changing at each passage.
\end{proof}

\noindent
{\bf Question.} Do exist a bounded domain $\Omega$ of class $C^2$ and a weight $w\in C^{0,1}(\overline{\Omega}) \cap A_p$ such that \eqref{LI} is not met?

\section*{Acknowledgments}
The first author acknowledges the Science and Engineering Research Board, Government of India, for National Postdoctoral Fellowship, file no. PDF/2023/000038. The second and the third author are member of the {\em Gruppo Nazionale per l'Analisi Ma\-te\-ma\-ti\-ca, la Probabilit\`a e le loro Applicazioni} (GNAMPA) of the {\em Istituto Nazionale di Alta Matematica} (INdAM); they are partially supported by the INdAM-GNAMPA Project 2024 ``Regolarità ed esistenza per operatori anisotropi'' (E5324001950001). \\
This study was partly funded by: Research project of MIUR (Italian Ministry of Education, University and Research) PRIN 2022 ``Nonlinear differential problems with applications to real phenomena'' (Grant Number: 2022ZXZTN2).

\end{document}